\documentclass[a4paper,10pt,twoside]{article}

\usepackage[margin=1.5in, marginparwidth=1.2in]{geometry}
\usepackage[numbers]{natbib}
\usepackage[T1]{fontenc}
\usepackage[utf8]{inputenc}

\usepackage{amsmath} 
\usepackage{mathrsfs}
\usepackage{amsthm} 
\usepackage{amsfonts}
\usepackage{amssymb, latexsym}
\usepackage{mhequ}
\usepackage{verbatim}

\usepackage{kerkis}

\usepackage{graphicx}
\usepackage{color}
\usepackage{xcolor}
\usepackage{tikz}
\usetikzlibrary{calc,intersections,through,backgrounds}
\usetikzlibrary{decorations.pathreplacing}
\usetikzlibrary{shapes}

\usepackage{float} 
\usepackage{enumerate}

\usepackage[pdfauthor={}, pdftitle={}, pdftex]{hyperref}
\hypersetup{
	colorlinks=false,
	pdfborder={0 0 0}
}

\usepackage{tocloft}
\usepackage{fancyhdr}

\usepackage[]{appendix}

\usepackage{hyperref}
\hypersetup{
	colorlinks=true,    
	linkcolor=blue,     
	citecolor=teal,   
	filecolor=blue,
	urlcolor=cyan,
}

\allowdisplaybreaks
\theoremstyle{plain}
\newtheorem{thm}{Theorem}[section]

\newtheorem{lem}[thm]{Lemma}

\numberwithin{equation}{section}
\theoremstyle{remark}
\newtheorem{remark}[thm]{Remark}

\theoremstyle{definition}
\newtheorem{definition}[thm]{Definition}

\newcommand{\N}{\mathbb{N}}
\newcommand{\R}{\mathbb{R}}
\newcommand{\Z}{\mathbb{Z}}

\renewcommand{\max}{\operatorname{max}}

\newcommand{\defeq}{\mathrel{:=}}
\renewcommand{\d}{\text{d}}
\newcommand{\norm}[1]{\left\lVert #1 \right\rVert} 

\title
{Global well-posedness of the Boltzmann equation via bilinear estimates}
\author{Anne Niesdroy}

\date{}

\fancypagestyle{plain}{

	\fancyhf{}
	\fancyfoot[C]{\thepage}
}

\renewenvironment{abstract}
{
	\begin{minipage}{.9\textwidth}\small\textbf{Abstract}.\noindent
	}
	{
	\end{minipage}
}

\newenvironment{keywords}
{
	\begin{minipage}{.9\textwidth}\small\textbf{Keywords}:\noindent
	}
	{
	\end{minipage}
}

\newenvironment{msc}
{
	\begin{minipage}{.9\textwidth}\small\textbf{MSC 2020}:\noindent
	}
	{
	\end{minipage}
}

\begin{document}
	\maketitle	
	
	\begin{center}
		\begin{abstract}
			We prove global well-posedness of the hard-sphere Boltzmann equation in $\mathbb{R}^d$ for small initial data in the critical Besov space $B_{1,1}^{d-1}L_v^1$ and in the critical Sobolev space $W^{d-1,1}L_v^1$. The proof is based on new bilinear estimates for the nonlinear collision operator, which rely on transversality considerations. 
	\end{abstract}	
	\end{center}
	\bigskip

\smallskip

\begin{msc} 35Q20 \end{msc}

	\tableofcontents
	\section{Introduction}
We consider the Cauchy problem for the Boltzmann equation
\begin{align}
	\partial_t f + v \cdot \nabla_x f &= Q(f,f)\tag{IVP}\label{IVP},\\
	f(0,x,v) &= f^0(x,v)\notag,
\end{align}
within the domain $(0, \infty)\times \R^d \times \R^d$. The collision operator $Q(f,f)$ is given by
\begin{align*}
	Q(f,f)(v) = \int_{\R^d}\int_{\omega \in S_+^{d-1}} \left(f(t,x,v^{\ast})f(t,x,u^{\ast})-f(t,x,v)f(t,x,u)\right)B(v-u,\omega)\d \omega \d u,
\end{align*}
where the post-collision velocities are defined by
\begin{align*}
	v^{\ast} = v - [\omega\cdot(v-u)]\omega, \qquad u^{\ast} = u + [\omega\cdot(v-u)]\omega,
\end{align*}
 and $\omega \in S_+^{d-1}$, where $S_+^{d-1}\defeq\{\omega\in S^{d-1}:\omega^{(d)}>0\}$ is a fixed hemisphere. Since $v^\ast$ and $u^\ast$ do not change under $\omega\mapsto-\omega$, every collision is counted exactly once.
For two functions $f,g$ we denote by
\begin{align*}
	Q(f,g)(v) = \int_{\R^d}\int_{\omega \in S_+^{d-1}} \left(f(t,x,v^{\ast})g(t,x,u^{\ast})-f(t,x,v)g(t,x,u)\right)B(v-u,\omega)\d \omega \d u
\end{align*}
the associated (non-symmetrized) bilinear operator, so that $Q(f,f)$ is recovered for $g=f$.
 
In this work, we consider the cutoff hard-sphere model, where the collision kernel $B$ factorizes as
\begin{align*}
	B(v-u, \omega) = \left \vert v- u \right \vert^{\gamma} b(\cos (\theta)), 
\end{align*}
with $\gamma = 1 $, $\cos (\theta) = \left \langle \omega, \frac{(v-u)}{\left \vert v-u \right \vert }\right \rangle$ and $0 \leq b(\cos (\theta)) \leq C \left \vert \cos (\theta)\right \vert$.
Under the cutoff assumption, the operator $Q$ can be decomposed into a gain term $Q^+$ and a loss term $Q^-$.

The nonlinear collision term contains the product of the distribution function at
one space point evaluated at two different velocities.
This structure typically forces well-posedness results based on fixed-point arguments in the physical variables to be established in spaces embedded into $L_x^\infty$, or, alternatively, requires dispersive (Strichartz-type) estimates for the kinetic transport operator, which are only available in $L^p$-based spaces with $p>1$.
Our objective is to overcome this restriction and prove global well-posedness in critical $L^1$-based spaces without relying on full spatial $L^\infty$-control. We only use $L^\infty$-bounds in the $d-1$ spatial directions transversal to the relative velocity (and none if $d=1$).

Our approach is inspired by transversality arguments from dispersive PDE theory. This includes the use of bilinear estimates for nonlinear dispersive equations and the transference of estimates for free solutions to suitable adapted function spaces  (see e.g. \cite{KM1993,B1988,KF2000,T2007,CH2018}).
First, we establish bilinear estimates for free transport solutions.
Then, we introduce new solution spaces that capture suitable superpositions of free flows.
This structure allows us to transfer bilinear bounds from free solutions to nonlinear solutions.
A contraction argument in these spaces yields global well-posedness for small initial data.

We work at minimal spatial integrability, namely in $L^1$, and obtain the following small-data result at critical regularity.

\begin{thm}[Global well-posedness for small initial data]\label{thm_glob_well_posedness}
 Let $s \geq d-1$.
 \begin{enumerate}
 	\item If $d=1$, assume in addition $s>0$. There exists $\varepsilon_0 >0$ such that if the initial data $f^0$ is nonnegative and $\norm{f^0}_{B_{1,1}^sL_v^1} < \varepsilon_0$, then there exists a unique nonnegative global mild solution $f$ of \eqref{IVP} satisfying $Q(f,f) \in L^1([0,\infty),B_{1,1}^sL_v^1)$.
 	\item If in addition $s \in \N$, there exists $\varepsilon_0 >0$ such that if the initial data $f^0$ is nonnegative and $\norm{f^0}_{W^{s,1}L_v^1} < \varepsilon_0$, then there exists a unique nonnegative global mild solution $f$ of \eqref{IVP} satisfying $Q(f,f) \in L^1([0,\infty),W^{s,1}L_v^1)$.
 \end{enumerate}
In both cases, let $\mathcal{X}^s$ denote $B^s_{1,1}L^1_v$ resp.\ $W^{s,1}L^1_v$, and let $(U(t)\varphi)(x,v)\defeq\varphi(x-tv,v)$ denote the free transport flow. Then, the solution $f$ has the following properties.
\begin{enumerate}[(i)]
	\item \emph{Continuity in time.} $f\in C([0,\infty),\mathcal{X}^s)$ and
	\begin{align*}
		\sup_{t\ge0}\norm{f(t)}_{\mathcal{X}^s}\le\norm{f^0}_{\mathcal{X}^s}+\norm{Q(f,f)}_{L^1_t\mathcal{X}^s}.
	\end{align*}
	\item \emph{Lipschitz continuous dependence on the initial data.} If $\tilde f$ denotes the solution with nonnegative initial data $\tilde f^0$ and $\norm{\tilde f^0}_{\mathcal{X}^s}<\varepsilon_0$, then
	\begin{align*}
		\sup_{t\ge0}\norm{f(t)-\tilde f(t)}_{\mathcal{X}^s}+\norm{Q(f,f)-Q(\tilde f,\tilde f)}_{L^1_t\mathcal{X}^s}\le 3\norm{f^0-\tilde f^0}_{\mathcal{X}^s}.
	\end{align*}
	\item \emph{Scattering.} There exists $f_\infty\in\mathcal{X}^s$ such that
	\begin{align*}
		\lim_{t\to\infty}\norm{f(t)-U(t)f_\infty}_{\mathcal{X}^s}=0.
	\end{align*}
\end{enumerate}
\end{thm}
\begin{remark}
Writing $x = (x^{(i)},x')$ with $x^{(i)}\in\R$ and $x'\in\R^{d-1}$ for $i\in\{1,\dots,d\}$, the proof relies on the anisotropic embeddings
\begin{align*}
	B_{1,1}^{d-1}\left(\R^{d}\right) \hookrightarrow L_{x^{(i)}}^1L_{x'}^{\infty}, \qquad W^{d-1,1}\left(\R^{d}\right) \hookrightarrow L_{x^{(i)}}^1L_{x'}^{\infty},
\end{align*}
see Lemma~\ref{lem_anisotropic}. In the Sobolev case they follow from the embedding $W^{d-1,1}(\R^{d-1})\hookrightarrow L^\infty(\R^{d-1})$, applied for a.e.\ $x^{(i)}$, and in the Besov case from a Bernstein-type estimate. Thus, $d-1$ derivatives control the $d-1$ directions $x'$, whereas in the direction $x^{(i)}$ no $L^\infty$-bound is needed, since there the transversality of the free flows is used.
In contrast to the recent preprint of Hu, Liu, Peng and Zhou \cite{HLZ2026}, our approach avoids the use of additional velocity weights and therefore yields a genuinely critical result.
Here, criticality refers to the scaling $f_\lambda(t,x,v)=\lambda f(\lambda t,\lambda x,v)$, $\lambda>0$, which leaves the hard-sphere equation \eqref{IVP} invariant and under which the homogeneous norm $\norm{f_\lambda(0)}_{\dot B^{s}_{1,1}L^1_v}=\lambda^{1+s-d}\norm{f(0)}_{\dot B^{s}_{1,1}L^1_v}$ is invariant if and only if $s=d-1$.
\end{remark}

\begin{remark}
We expect that the bilinear estimates of Section~\ref{sec_bilinear} extend to more general hard potentials $0\le\gamma<1$ in $L^a$-based spaces with suitable polynomial velocity weights. This will be addressed in the forthcoming thesis of the author.
\end{remark}

\subsection{Comments on the literature}\label{sec_lit}
The Cauchy problem for the Boltzmann equation has been a central topic in kinetic theory for several decades. The main objective is to establish global well-posedness under the weakest possible assumptions on the initial data. Various settings have been investigated, including cutoff and non-cutoff assumptions on the collision kernel, as well as hard and soft potentials and Maxwellian-type interactions.

In the following, we provide a brief and non-exhaustive overview of the literature concerning existence results for the Cauchy problem.

General background on the Boltzmann equation can be found in \cite{C1988} and \cite{V2002}.

A major breakthrough was achieved by DiPerna and Lions \cite{DL1989b}, who treated arbitrarily large initial data. They introduced the concept of renormalized solutions and proved their global-in-time existence. However, uniqueness of renormalized solutions remains an open problem.

To address uniqueness, several works have focused on special classes of initial data. One important class consists of perturbations of the Maxwellian equilibrium, where the initial data are sufficiently close to a Maxwellian distribution (cf. \cite{DLX2016} and the references therein). 

In this paper, we focus instead on solutions corresponding to small initial data, often referred to as the near-vacuum regime. This line of research goes back to Illner and Shinbrot \cite{IS1984}, who constructed weighted $L^{\infty}$-solutions for the hard-sphere model using a contraction mapping argument. Later, Arsénio \cite{A2011} established global existence of mild solutions for small initial data in $L^D$-spaces, where $D$ denotes the dimension, using dispersive (Strichartz-type) estimates for the kinetic transport equation. More recently, Chen, Denlinger and Pavlović \cite{CDP2021} developed a bilinear spacetime estimate framework to prove small-data global well-posedness for the Boltzmann equation with constant collision kernels.

A second line of research is based on the dispersive properties of the free transport flow. For cutoff soft potentials $-1\le\gamma<0$ in dimension three, He and Jiang \cite{HJ2017,HJ2023} proved global well-posedness and scattering for small initial data in $L^p_{x,v}$, combining Strichartz estimates for kinetic transport \cite{CP1996,O2011} with the $L^p$-estimates for the gain term of Alonso, Carneiro and Gamba \cite{ACG2010}. For $0<\gamma\le1$ the latter estimates require higher velocity weights on the right-hand side than on the left-hand side, so that the fixed-point argument does not close. This obstruction was removed by He, Jiang, Kuo and Liang \cite{HJKL2024}, who proved Hardy--Littlewood--Sobolev type estimates for $Q^+$ without loss of velocity weights for $0\le\gamma\le1$. As an application, they obtained global well-posedness and scattering in $\R^3$ for initial data which are small in a weighted $L^3_{x,v}$-space and satisfy an additional integrability condition.

In contrast to \cite{HJKL2024}, Theorem~\ref{thm_glob_well_posedness} is formulated in scaling-critical $L^1$-based spaces without velocity weights, at the expense of $d-1$ spatial derivatives. Neither class of initial data is contained in the other. Instead of Strichartz estimates, our proof uses the transversality of the free flows, which compensates the factor $\vert v-u\vert$ of the hard-sphere kernel (Lemma~\ref{lem: bilinear_free}). This elementary argument treats the gain and the loss term in the same way and works in every dimension $d\ge1$.

In the one-dimensional case $d=1$, Ha \cite{H2003}, Benedetto, Caglioti and Pulvirenti \cite{BCP1997} and Benedetto and Pulvirenti \cite{BP2001} proved existence and uniqueness of mild solutions to the one-dimensional inelastic Boltzmann equation for nonnegative, bounded initial data with compact support in the velocity variable.
More recently, Chen, Shen and Zhang \cite{CSZ2023,CSZ2024} investigated both well-posedness and ill-posedness in critical Sobolev spaces using a dispersive approach, for Maxwellian molecules and soft potentials. They introduced scaling-critical $L^2$-based superposition spaces. However, their method is based on Fourier analysis in the $L^2$-framework. Very recently, Hu, Liu, Peng and Zhou \cite{HLZ2026} studied weighted anisotropic Besov spaces and obtained results comparable to ours for both hard and soft potentials. In contrast to our approach, their method requires additional weights to control the transport structure. By exploiting the superposition structure of free solutions, we avoid this difficulty and directly obtain scale-invariant critical regularity.

A conceptually related idea was already employed by Cercignani in the context of the Enskog equation \cite{CC1988}, where the decoupling of the spatial variables simplifies the analysis.

\subsection{Structure of the paper}
In Section~2 we first introduce the functional framework and show some preliminary boundedness results for Littlewood-Paley projections. In Section~\ref{sec_bilinear} we establish the required bilinear estimates for free transport solutions and transfer them to the solution space $\mathcal{N}^s$ via a superposition principle. Finally, in Section~4 we prove global well-posedness for small initial data and scattering.

 \section{Notation and preliminaries}
For two quantities $A$ and $B$, $A \lesssim B$ means that there is a generic constant $C> 0$ such that $A \leq CB$. 
 \subsection{Fourier transform and Fourier multipliers}
 	For fixed $(t,v)$ we define the Fourier transform in the spatial variable $x$ of a Schwartz function $\varphi \in \mathcal{S}\left(\R^d\right)$ by
 	\begin{align}
 	\mathcal{F}(\varphi)(t,\xi,v):=\int_{\R^d}{e^{-2\pi i x\cdot\xi}\varphi(t,x,v)\d x}, \quad \xi \in \R^d.
 	\end{align}
 For a tempered distribution $f \in \mathcal{S}'\left(\R^d\right)$ the Fourier transform is defined in the usual dual sense  
 	\begin{align*}
 		\langle \mathcal{F}(f), \varphi\rangle =\langle f, \mathcal{F}(\varphi)\rangle.
 	\end{align*}
Its inverse is given by
\begin{align*}
	\mathcal{F}^{-1}(\psi)(t,x,v)=\int_{\R^d}e^{2\pi i x\cdot\xi}\psi(t,\xi,v)\d\xi,\quad x\in\R^d,
\end{align*}
and extended to $\mathcal{S}'(\R^d)$ in the same way. With this normalization we have $\mathcal{F}(\varphi\ast\psi)=\mathcal{F}(\varphi)\mathcal{F}(\psi)$ without additional constants. To emphasize the variable, we also write $\mathcal{F}_x$ and $\mathcal{F}_x^{-1}$.
Next, we introduce a dyadic partition of unity in $\R^d$. Let $\eta_0$ be a smooth function with values in $[0, 1]$ such that $\eta_0$ is supported in the ball with radius $\frac{8}{5}$
and satisfies $\eta_0 \equiv 1$ on $B_{4/5}(0)$. For $k \in \Z$ define
\begin{align*}
	\chi_k(\xi) := \eta_0(2^{-k}\xi) - \eta_0(2^{-k+1}\xi),
	\qquad
	\chi_{\le k}(\xi) := \eta_0(2^{-k}\xi).
\end{align*}
The associated Littlewood-Paley projections are given by
\begin{align*}
	P_k f := \mathcal{F}_x^{-1} \left( \chi_k \mathcal{F}_x (f) \right),
	\qquad
	P_{\le k} f := \mathcal{F}_x^{-1} \left( \chi_{\le k} \mathcal{F}_x (f) \right).
\end{align*}

We next establish uniform $L^p$-boundedness of the Littlewood-Paley projections $P_k$ and $P_{\leq k}$.

\begin{lem}\label{lem_bound_P}
	Let  $1\leq p\leq \infty$ and $k \geq 0$. Then, there is a constant $C>0$, independent of $k$ such that for all $f \in  L^p$ it holds that
	\begin{equation}
		\norm{P_k f}_{L_x^p}\leq  C\norm{f}_{L_x^p},\qquad  \norm{P_{\leq k}f}_{L_x^p} \leq C\norm{f}_{L_x^p}.
	\end{equation}
\end{lem}

\begin{proof}
	By definition $P_k f = \mathcal{F}^{-1}\left(\chi_k\right)\ast f$. Since $\chi_k(\xi)=\chi_0\left(2^{-k}\xi\right)$, we have the scaling relation $\mathcal{F}^{-1}\left(\chi_k\right)(x) = 2^{kd} \mathcal{F}^{-1}(\chi_0)\left(2^k x\right)$. Hence, $\norm{\mathcal{F}^{-1}(\chi_k)}_{L^1} = \norm{\mathcal{F}^{-1}(\chi_0)}_{L^1}$, independently of $k$. Young's inequality therefore yields 
	\begin{align*}
		\norm{P_k f}_{L_{x}^p}  \leq \norm{\mathcal{F}^{-1}\left(\chi_k\right)}_{L_{x}^1}\norm{f}_{L_{x}^p}\leq C \norm{f}_{L_x^p}.
	\end{align*}
	The estimate for $P_{\le k}$ follows in the same way.
\end{proof}

\begin{remark}\label{rem_mixed}
Since $P_k$ and $P_{\le k}$ are convolutions with kernels in $L^1(\R^d)$, Minkowski's integral inequality shows that the bounds of Lemma~\ref{lem_bound_P} also hold in the mixed norms $L^1_{x^{(i)}}L^\infty_{x'}$, for any splitting $x=(x^{(i)},x')$ of the coordinates, with the same constant. Moreover, $P_k$ commutes with translations in $x$.
\end{remark}

\begin{remark}[The case $d=1$]\label{rem_d1}
Throughout, for a splitting $x=(x^{(i)},x')$ we write $x'\in\R^{d-1}$ for the remaining $d-1$ coordinates, and analogously $v'$, $u'$. If $d=1$, there are no such coordinates. In this case we use the convention that $\R^{d-1}=\R^0$ is a single point equipped with the counting measure. Then all norms in $x'$ reduce to the absolute value, in particular $L^1_{x^{(1)}}L^\infty_{x'}=L^1_x$ and $L^1_\tau L^{q}_{x'}=L^1_\tau$ for all $q\in[1,\infty]$, and the embeddings for function spaces on $\R^{d-1}$ used below are trivial. Moreover, $S^{0}_+$ consists of a single point, the integral $\d\omega$ is the counting measure, and $\vert v-u\vert=\vert v^{(1)}-u^{(1)}\vert$. With this convention, all proofs below apply verbatim for $d=1$, with $s\ge d-1=0$ and $s>0$ in the Besov case.\end{remark}

 \subsection{Function spaces and embeddings}\label{sec_function_spaces}
Let $1 \leq p \leq \infty$, and let $k \geq 0$ be a natural number. A function $f$ is said to lie in the classical Sobolev space $W^{k,p}\left({\R}^d\right)$ if its weak derivatives $D^{\alpha} f$ exist and lie in $L^p\left({\R}^d\right)$ for all multi-indices $\alpha$ with $\vert \alpha \vert \leq k$. If $f$ lies in $W^{k,p}\left({\R}^d\right)$, we define the $W^{k,p}$ norm of $f$ by 
\begin{align*}
	\norm{f}_{W^{k,p}\left({\R}^d\right)} := \sum_{\vert \alpha \vert \leq k} \norm{D^{\alpha} f}_{L^p\left({\R}^d\right)}.
\end{align*}
  The inhomogeneous Besov space is defined by 
  		\begin{align*}
 			B^s_{p,q}\left(\R^d\right):=\{f \in \mathcal{S}'\left(\R^d\right): \|f\|_{B^s_{p,q}}< \infty\},
 		\end{align*}
where 
 		\begin{align*}
 			\|f\|_{B^s_{p,q}} :=  \begin{cases}
 				\|P_{\leq 0}(f)\|_{L^p}+ \left(\sum_{j=1}^{\infty}\left(2^{js}\|P_j(f)\|_{L^p}\right)^q \right)^{\frac{1}{q}}&, \quad q<\infty,\\
 				\|P_{\leq 0}(f)\|_{L^p}+ \sup_{j\geq 1}2^{js}\|P_j(f)\|_{L^p}&, \quad q=\infty.
 			\end{cases}
 		\end{align*}
Let $s \geq 0$ and $\mathcal{X}^s=B_{1,1}^sL_v^1$ or, if $s \in \N$, $\mathcal{X}^s \in \{B^s_{1,1}L_v^1, W^{ s  ,1}L_v^1\}$. If $d=1$ and $\mathcal{X}^s=B^s_{1,1}L^1_v$, we assume in addition $s>0$. Here and below, $\norm{h}_{B^s_{1,1}L^1_v}\defeq\int_{\R^d}\norm{h(\cdot,v)}_{B^s_{1,1}}\d v$ and analogously for $W^{s,1}L^1_v$. For $p=q=1$ the order of the $v$-integration and the dyadic sum may be interchanged by Fubini's theorem. Let $A\defeq\partial_t+v\cdot\nabla_x$ denote the transport operator and $(U(t)\varphi)(x,v)\defeq\varphi(x-tv,v)$ the free transport flow. Since translations act isometrically and continuously on $B^s_{1,1}(\R^d)$ and $W^{s,1}(\R^d)$, $(U(t))_{t\in\R}$ is a strongly continuous group of isometries on $\mathcal{X}^s$.
We define the solution space  $\mathcal{N}^s$ by
\begin{align*}
	\mathcal{N}^s\defeq\Big\{f \in C([0,\infty),\mathcal{X}^s) : Af \in L^1\left((0,\infty),\mathcal{X}^s\right)\Big\}
\end{align*}
with norm
\begin{align*}
	\norm{f}_{\mathcal{N}^s}\defeq\norm{f(0)}_{\mathcal{X}^s}+\norm{Af}_{L^1_{t}\mathcal{X}^s},
\end{align*}
where $Af$ is understood in the sense of distributions on $(0,\infty)\times\R^d\times\R^d$.
For $f\in\mathcal{N}^s$ let $g(t)\defeq U(-t)f(t)$, i.e.\ $g(t,y,v)=f(t,y+tv,v)$. By the chain rule,
\begin{align*}
	\partial_tg(t,y,v)=(\partial_tf+v\cdot\nabla_xf)(t,y+tv,v)=(Af)(t,y+tv,v),
\end{align*}
and this identity also holds in the sense of distributions. Thus, $\partial_tg=U(-t)(Af)(t)$ with $U(-\cdot)(Af)(\cdot)\in L^1((0,\infty),\mathcal{X}^s)$. Integrating in time and applying $U(t)$ yields Duhamel's formula
\begin{align}\label{eq_duhamel_abstract}
	f(t)=U(t)f(0)+\int_0^tU(t-r)(Af)(r)\d r\qquad\text{for all }t\ge0,
\end{align}
where the integral is a Bochner integral in $\mathcal{X}^s$. Pointwise, this means
\begin{align}\label{eq_duhamel_N}
	f(t,x,v)=f(0)(x-tv,v)+\int_0^t(Af)(r,x-(t-r)v,v)\d r\qquad\text{for a.e. }(t,x,v).
\end{align}
Conversely, for all $\varphi\in\mathcal{X}^s$ and $h\in L^1((0,\infty),\mathcal{X}^s)$ the function $f(t)\defeq U(t)\varphi+\int_0^tU(t-r)h(r)\d r$ belongs to $\mathcal{N}^s$ with $f(0)=\varphi$ and $Af=h$.

 Since \eqref{IVP} is posed for $t>0$, we work with $t\ge0$ throughout this paper. In particular, $L_t^1$ denotes $L^1((0,\infty))$ in time.
 \begin{lem}\label{lem_banach}
 	The solution space $\mathcal{N}^s$ is a Banach space.
 \end{lem}

 \begin{proof}
 	By Duhamel's formula \eqref{eq_duhamel_abstract} and its converse, the map $(\varphi,h)\mapsto f$ is a linear bijection from $\mathcal{X}^s\times L^1_t\mathcal{X}^s$ onto $\mathcal{N}^s$, which is isometric if the product is equipped with the norm $\norm{\varphi}_{\mathcal{X}^s}+\norm{h}_{L^1_t\mathcal{X}^s}$. Since $\mathcal{X}^s$ and $L^1_t\mathcal{X}^s$ are Banach spaces, so is $\mathcal{N}^s$.
 \end{proof}

\begin{lem}\label{lem_sup_bound}
For all $f\in\mathcal{N}^s$ we have $\sup_{t\ge0}\norm{f(t)}_{\mathcal{X}^s}\le\norm{f}_{\mathcal{N}^s}$.
\end{lem}
\begin{proof}
By \eqref{eq_duhamel_abstract}, $f(t)=U(t)f(0)+\int_0^tU(t-r)(Af)(r)\d r$. Since $U(t)$ is an isometry on $\mathcal{X}^s$, Minkowski's inequality gives the claim.
\end{proof}

Later, we will use the following Sobolev embedding theorem in dimension $d-1$. In the case $d=1$ it is not needed, since all norms in $x'$ are trivial by Remark~\ref{rem_d1}.

\begin{thm}[Sobolev embedding]\label{thm_sobolev_embedding}\cite[Theorem 4.12]{AF2003}
Let $1\le p<\infty$ and $l,k\in\N_0$ with $l\le k$. Then $W^{k,p}(\R^d)\hookrightarrow W^{l,q}(\R^d)$ in each of the following cases.
\begin{enumerate}[(a)]
	\item $(k-l)p>d$, or $k-l=d$ and $p=1$, and $p\le q\le\infty$.
	\item $(k-l)p=d$ and $p\le q<\infty$.
	\item $(k-l)p<d$ and $p\le q\le\frac{dp}{d-(k-l)p}$.
\end{enumerate}
\end{thm}

\begin{lem}\label{lem_anisotropic}
Let $d\ge1$, $s\ge d-1$, and write $x=(x^{(i)},x')$ with $x'\in\R^{d-1}$ for some $i\in\{1,\dots,d\}$. Then there is $C>0$ such that
\begin{align*}
	\norm{h}_{L^1_{x^{(i)}}L^\infty_{x'}}\le C\norm{h}_{B^{s}_{1,1}(\R^d)},\qquad \norm{h}_{L^1_{x^{(i)}}L^\infty_{x'}}\le C\norm{h}_{W^{s,1}(\R^d)} \ (s\in\N).
\end{align*}
\end{lem}
\begin{proof}
We start with the Besov estimate. Since $\chi_{\le k+1}\equiv1$ on the support of $\chi_k$, we have $P_kh=\mathcal{F}^{-1}(\chi_{\le k+1})\ast P_kh$ for $k\ge1$, and analogously $P_{\le0}h=\mathcal{F}^{-1}(\chi_{\le1})\ast P_{\le0}h$. As in the proof of Lemma~\ref{lem_bound_P}, we have the scaling relation $\mathcal{F}^{-1}(\chi_{\le k})(x)=2^{kd}\mathcal{F}^{-1}(\eta_0)(2^kx)$. Hence, the substitution $y=2^kx^{(i)}$ yields
\begin{align*}
	\norm{\mathcal{F}^{-1}(\chi_{\le k})}_{L^1_{x^{(i)}}L^\infty_{x'}}=2^{k(d-1)}\norm{\mathcal{F}^{-1}(\eta_0)}_{L^1_{x^{(i)}}L^\infty_{x'}},
\end{align*}
which is finite, since $\mathcal{F}^{-1}(\eta_0)$ is a Schwartz function. Minkowski's inequality therefore yields the Bernstein-type estimate
\begin{align*}
	\norm{P_kh}_{L^1_{x^{(i)}}L^\infty_{x'}}\le\norm{\mathcal{F}^{-1}(\chi_{\le k+1})}_{L^1_{x^{(i)}}L^\infty_{x'}}\norm{P_kh}_{L^1_x}\le C2^{k(d-1)}\norm{P_kh}_{L^1_x},
\end{align*}
and in the same way $\norm{P_{\le0}h}_{L^1_{x^{(i)}}L^\infty_{x'}}\le C\norm{P_{\le0}h}_{L^1_x}$. Since $s\ge d-1$, we conclude
\begin{align*}
	\norm{h}_{L^1_{x^{(i)}}L^\infty_{x'}}\le\norm{P_{\le0}h}_{L^1_{x^{(i)}}L^\infty_{x'}}+\sum_{k=1}^{\infty}\norm{P_kh}_{L^1_{x^{(i)}}L^\infty_{x'}}\\
	&\le C\Big(\norm{P_{\le0}h}_{L^1_x}+\sum_{k=1}^{\infty}2^{ks}\norm{P_kh}_{L^1_x}\Big)\\
	&=C\norm{h}_{B^s_{1,1}}.
\end{align*}
For the Sobolev estimate, let $s\in\N$. By Fubini's theorem, $h(x^{(i)},\cdot)\in W^{s,1}(\R^{d-1})$ for a.e.\ $x^{(i)}\in\R$, and
\begin{align*}
	\int_{\R}\norm{h(x^{(i)},\cdot)}_{W^{s,1}(\R^{d-1})}\d x^{(i)}\le\norm{h}_{W^{s,1}(\R^d)}.
\end{align*}
Since $s\ge d-1$, Theorem~\ref{thm_sobolev_embedding} yields $\norm{h(x^{(i)},\cdot)}_{L^\infty(\R^{d-1})}\le C\norm{h(x^{(i)},\cdot)}_{W^{s,1}(\R^{d-1})}$. Integrating in $x^{(i)}$ gives the second estimate.
\end{proof}

 \section{Bilinear estimates}\label{sec_bilinear}
In this section we establish the key bilinear estimates for the collision operator. We begin with the case of free transport solutions, where the underlying transversality structure becomes transparent. The general case will then follow by a superposition argument.

\begin{lem}\label{lem: bilinear_free}
Let $s \geq d-1$.	There exists $C>0$ such that for all free transport solutions
$$f(t,x,v) = f^0(x-tv,v), \qquad g(t,x,u) = g^0(x-tu,u)$$ with $f^0, g^0 \in \mathcal{X}^s$ the collision operator satisfies $Q(f,g) \in {L_t^1\mathcal{X}^s}$ and
	\begin{align*}
	\norm{Q(f,g)}_{L_t^1\mathcal{X}^s} \leq C\norm{f^0}_{\mathcal{X}^s}\norm{g^0}_{\mathcal{X}^s}.
	\end{align*}
	\end{lem}

\begin{proof}
Due to the cutoff assumption, the gain and loss terms can be treated separately. Since, after a change of variables, both contributions are estimated in the same way, it suffices to consider the gain term $Q^+$. Indeed, for fixed $\omega$ the map $(v,u)\mapsto(v^\ast,u^\ast)$ is a linear involution with Jacobian of modulus one which preserves $\vert v-u\vert$ and $\vert\cos\theta\vert$, so that after this change of variables $Q^+$ is bounded by the same expression as $Q^-$, with $\vert b(\cos\theta)\vert\le C\vert\cos\theta\vert$ and $\int_{S^{d-1}_+}\vert\cos\theta\vert\d\omega\le C$.

\textbf{Case 1: $\mathcal{X}^s = W^{s,1}L_v^1$, $s \in \N$.}
The weak derivatives $D_x^\alpha$ act only in $x$ and therefore commute with the integration in $(\omega,u)$. Thus, by a change of variables, Fubini's theorem and Leibniz' rule, 
	\begin{align*}
	&\norm{Q^+(f,g)}_{L_t^1W^{s,1}L_v^1} \\
	&\leq \sum_{\vert \alpha \vert \leq  s  }\int_{0}^{\infty} \int_{\R^d}  \int_{\R^d}\int_{\omega \in S_+^{d-1}}\int_{\R^d}\left\vert D_x^{\alpha}\left(f(t,x,v^{\ast})g(t,x,u^{\ast})\right)  \right\vert \left\vert B(v-u,\omega)\right\vert\d x \d \omega \d u  \d v  \d t\\
	&=  \sum_{\vert \alpha \vert \leq  s  }\int_{0}^{\infty} \int_{\R^d}  \int_{\R^d}\int_{\omega \in S_+^{d-1}}\int_{\R^d}\left\vert D_x^{\alpha}\left (f(t,x,v)g(t,x,u)\right)\right \vert \left \vert B(v-u,\omega)\right\vert\d x \d \omega \d u  \d v  \d t\\
	&\lesssim  \sum_{\vert \alpha + \beta \vert \leq  s  } \int_{0}^{\infty} \int_{\R^d}  \int_{\R^d}\int_{\omega \in S_+^{d-1}}\int_{\R^d}\left \vert D_x^{\alpha}\left(f(t,x,v)\right)D_x^{\beta}\left(g(t,x,u)\right)  \right \vert \left \vert u-v \right\vert  \left \vert b(\cos(\theta)) \right \vert\d x \d \omega \d u  \d v  \d t.
\end{align*}
Using the representation of free solutions $f(t,x,v) = f^0 (x-tv,v), g(t,x,u) = g^0 (x-tu, u )$, we perform the change of variables $x \mapsto x + tu$, which yields
\begin{align*}
	&\int_{0}^{\infty} \int_{\R^d}  \int_{\R^d}\int_{\omega \in S_+^{d-1}}\int_{\R^d}\left \vert D_x^{\alpha}\left(f(t,x,v)\right)D_x^{\beta}\left(g(t,x,u)\right)  \right\vert \left \vert u-v \right\vert  \left \vert b(\cos(\theta)) \right \vert\d x \d \omega \d u  \d v  \d t\\
	&=\int_{0}^{\infty} \int_{\R^d}  \int_{\R^d}\int_{\omega \in S_+^{d-1}}\int_{\R^d}\left \vert D_x^{\alpha}\left(f^0(x-t(v-u),v)\right)D_x^{\beta}\left(g^0(x,u)\right)  \right \vert \left \vert u-v \right \vert  \left \vert b(\cos(\theta)) \right \vert \d x \d \omega \d u  \d v  \d t.
\end{align*}
The key observation is that the interaction between the two free flows
is transversal unless $v=u$. Given $u,v \in \R^d$, we choose the spatial direction, where the componentwise difference is maximal. The calculation will be the same for all spatial directions and will work for the same leading constant. More precisely, we split $\R^d\times\R^d=\bigcup_{i=1}^d\Omega_i$, where $\Omega_i$ is the set of $(v,u)$ with $\vert v^{(i)}-u^{(i)}\vert=\max_l\vert v^{(l)}-u^{(l)}\vert$, and estimate the integral over each $\Omega_i$ separately. Without loss of generality we may assume that $i=1$, i.e.
$\vert v^{(1)} - u^{(1)}\vert = \max_i \vert v^{(i)}- u^{(i)} \vert $ and therefore, $\vert v-u \vert \leq \sqrt{d} \vert v^{(1)}- u^{(1)} \vert$.

We introduce the new variable  $\tau = x^{(1)}-t(v^{(1)}-u^{(1)})$, which corresponds to integration along the characteristic direction of the
relative velocity. The Jacobian of this transformation produces the factor
$|v^{(1)}-u^{(1)}|^{-1}$, which compensates the factor $\vert v-u \vert $.
By $x' \defeq (x^{(2)},\dots, x^{(d)})$ and $v'$, $u'$ respectively, we denote the $(d-1)$-dimensional vectors containing all components except the first one. For $d=1$ these variables are void and the norms in $x'$ are omitted, see Remark~\ref{rem_d1}. Furthermore, let $t^{\ast}(\tau,x^{(1)},v^{(1)},u^{(1)}) \defeq \frac{x^{(1)}-\tau}{v^{(1)}-u^{(1)}}$. Using Hölder's inequality we obtain
\begin{align*}
	& \int_{\R^d}  \int_{\R^d}\int_{\omega \in S_+^{d-1}}\int_{\R^d}\int_{0}^{\infty}\left \vert D_x^{\alpha}\left(f^0(x-t(v-u),v)\right)D_x^{\beta}\left(g^0(x,u)\right)  \right\vert \left \vert u-v \right \vert  \vert b(\cos(\theta)) \vert \d t \d x \d \omega  \d u  \d v \\
	&\leq \int_{\R^d}  \int_{\R^d}\int_{\omega \in S_+^{d-1}}\int_{\R^d} \int_{\R} \left \vert D_x^{\alpha} f^0(\tau,x'-t^{\ast}(v'-u'),v) D_x^{\beta}\left(g^0(x,u)\right) \right \vert \\
	&\vert u-v \vert  \left \vert u^{(1)}-v^{(1)} \right \vert^{-1} \vert b(\cos(\theta))\vert \d \tau \d x \d \omega  \d u  \d v \\
	&\leq C \norm{D^\alpha f^0(\tau,x',v)}_{L_{v}^1L^1_\tau L_{x'}^{q_\alpha}} \norm{D^\beta g^0(x,u)}_{L_{u}^{1}L^1_{x^{(1)}}L_{x'}^{q_\beta}},
\end{align*}
where $q_\alpha=\frac{\vert\alpha+\beta\vert}{\vert\alpha\vert}$, $q_\beta=\frac{\vert\alpha+\beta\vert}{\vert\beta\vert}$ are H\"older conjugate (with $q_\alpha=\infty$, $q_\beta=1$ if $\alpha=0$, including the case $\alpha=\beta=0$, and $q_\alpha=1$, $q_\beta=\infty$ if $\beta=0\neq\alpha$).
By the Sobolev embeddings $W_{x'}^{ s -\vert\alpha\vert ,1}\left(\R^{d-1}\right) \hookrightarrow L^{q_\alpha}_{x'}\left(\R^{d-1}\right)$ and  $W_{x'}^{ s-\vert\beta\vert  ,1}\left(\R^{d-1}\right)
\hookrightarrow L^{q_\beta}_{x'}\left(\R^{d-1}\right)$, applied for a.e.\ $\tau$, resp.\ $x^{(1)}$, together with Fubini's theorem, we conclude
\begin{align*}
\norm{Q(f,g)}_{L_t^1\mathcal{X}^s}\leq C\norm{f^0}_{W^{ s  ,1}L_{v}^{1}} \norm{g^0}_{W^{ s  ,1}L_{v}^{1}}.
\end{align*}
By Theorem~\ref{thm_sobolev_embedding} with $p=1$ in dimension $d-1$, the first embedding holds if $s-\vert\alpha\vert\ge(d-1)\frac{\vert\beta\vert}{\vert\alpha+\beta\vert}$. This follows from $\vert\alpha+\beta\vert\le s$ and $s\ge d-1$, since $(s-\vert\alpha\vert)\vert\alpha+\beta\vert-(d-1)\vert\beta\vert=\vert\alpha\vert(s-\vert\alpha+\beta\vert)+(s-d+1)\vert\beta\vert$. The second embedding follows in the same way with $\alpha$ and $\beta$ interchanged.

\textbf{Case 2: $\mathcal{X}^s = B_{1,1}^{s}L_v^1$}.
By definition of the Besov norm and, as in Case 1, by a change of variables, Minkowski's inequality and the cutoff bound $B(v-u,\omega) \lesssim \vert v - u \vert$, we obtain
\begin{align*}
	&\norm{Q^+(f,g)}_{L_t^1B_{1,1}^{s}L_{v}^1}\\
	&\lesssim \int_{0}^{\infty}\int_{\R^d}\int_{\R^d}\norm{P_{\leq 0}\big(f(t,\cdot,v)g(t,\cdot,u)\big)}_{L^1_x}\vert v-u\vert\d u\d v\d t\\
	&+\sum_{j=1}^{\infty}\int_{0}^{\infty}\int_{\R^d}\int_{\R^d}2^{js}\norm{P_j\big(f(t,\cdot,v)g(t,\cdot,u)\big)}_{L^1_x}\vert v-u\vert\d u\d v\d t.
\end{align*}
Using again the representation of free solutions $f(t,x,v) = f^0 (x-tv,v)$, $g(t,x,u) = g^0 (x-tu, u )$ and the change of variables $x \mapsto x + tu$, which commutes with $P_{\le0}$ and $P_j$, we replace $f(t,x,v)g(t,x,u)$ by $f^0(x-t(v-u),v)g^0(x,u)$. As in Case 1, it is enough to consider $(v,u)\in\Omega_1$. This leads to the two integrals
\begin{align*}
	I_1&\defeq\iint_{\Omega_1}\int_0^\infty\norm{P_{\le0}\big(f^0(\cdot-t(v-u),v)\,g^0(\cdot,u)\big)}_{L^1_x}\vert v-u\vert\d t\d u\d v,\\
	I_2&\defeq\iint_{\Omega_1}\int_0^\infty\sum_{j=1}^\infty2^{js}\norm{P_j\big(f^0(\cdot-t(v-u),v)\,g^0(\cdot,u)\big)}_{L^1_x}\vert v-u\vert\d t\d u\d v.
\end{align*}

The transversality enters through the following estimate. Let $(v,u)\in\Omega_1$. As in Case 1, the substitution $\tau=x^{(1)}-t(v^{(1)}-u^{(1)})$ for fixed $x$ and H\"older's inequality in $x'$ for fixed $(\tau,x^{(1)})$ yield
\begin{align}\label{eq_transversal}
	&\vert v-u\vert\int_0^\infty\int_{\R^d}\big\vert f^0(x-t(v-u),v)\big\vert\,\big\vert g^0(x,u)\big\vert\d x\d t\notag\\
	&\le\frac{\vert v-u\vert}{\vert v^{(1)}-u^{(1)}\vert}\int_{\R^d}\int_{\R}\big\vert f^0(\tau,x'-t^{\ast}(v'-u'),v)\big\vert\,\big\vert g^0(x,u)\big\vert\d\tau\d x\notag\\
	&\le\sqrt d\,\norm{f^0(\cdot,v)}_{L^1_{x^{(1)}}L^\infty_{x'}}\norm{g^0(\cdot,u)}_{L^1_x}.
\end{align}
After the change of variables $x\mapsto x+t(v-u)$ the same argument applies with the roles of $f^0$ and $g^0$ interchanged. Hence, the right-hand side of \eqref{eq_transversal} may also be replaced by $\sqrt d\,\norm{f^0(\cdot,v)}_{L^1_x}\norm{g^0(\cdot,u)}_{L^1_{x^{(1)}}L^\infty_{x'}}$. Since $P_k$ and $P_{\le k}$ commute with translations, both estimates also hold with $f^0$ replaced by $P_kf^0$ or $P_{\le k}f^0$ and $g^0$ replaced by $P_ig^0$ or $P_{\le i}g^0$.

We first estimate $I_1$. By Lemma~\ref{lem_bound_P} we have
\begin{align*}
	\norm{P_{\le0}\big(f^0(\cdot-t(v-u),v)\,g^0(\cdot,u)\big)}_{L^1_x}\le C\int_{\R^d}\big\vert f^0(x-t(v-u),v)\big\vert\,\big\vert g^0(x,u)\big\vert\d x.
\end{align*}
Thus, \eqref{eq_transversal}, Lemma~\ref{lem_anisotropic} and $\norm{g^0(\cdot,u)}_{L^1_x}\le\norm{g^0(\cdot,u)}_{B^s_{1,1}}$ yield
\begin{align*}
	I_1\le C\iint_{\Omega_1}\norm{f^0(\cdot,v)}_{L^1_{x^{(1)}}L^\infty_{x'}}\norm{g^0(\cdot,u)}_{L^1_x}\d u\d v\le C\norm{f^0}_{B_{1,1}^sL_{v}^{1}}\norm{g^0}_{B_{1,1}^sL_{v}^{1}}.
\end{align*}

For $I_2$ we apply the Littlewood-Paley decomposition in the spatial variable. With $f^0(\cdot-t(v-u),v)$ and $g^0(\cdot,u)$ in place of $f$ and $g$ we obtain
\begin{align*}
	fg &= \left(P_{\leq 0 }f +\sum_{k=1}^{\infty}P_k f\right)\cdot \left(P_{\leq 0 }g + \sum_{i=1}^{\infty}P_i g\right)\\
	&=P_{\leq 0}fP_{\leq 2 }g + P_{\leq 2}fP_{\leq 0 }g + \sum_{k=1}^{\infty}\sum_{\substack{i\geq 1,\\ \vert i-k\vert \leq 2}} P_{k}fP_{i}g
	 +\sum_{i=3}^{\infty} P_{\leq i-3}fP_i g  +\sum_{k=3}^{\infty} P_k f P_{\leq k-3}g.
\end{align*}
By Lemma~\ref{lem_anisotropic}, Lemma~\ref{lem_bound_P} and the support properties of the Littlewood-Paley projections we have $\norm{P_kf^0(\cdot,v)}_{L^1_{x^{(1)}}L^\infty_{x'}}\le C2^{ks}\norm{P_kf^0(\cdot,v)}_{L^1_x}$, and analogously for $g^0$. Together with \eqref{eq_transversal}, this shows that all series converge absolutely in $L^1_x$ for a.e.\ $t>0$. Hence, inserting the decomposition and by triangle inequality it remains to estimate the five corresponding integrals $I_{21},\dots,I_{25}$.

All five integrals are estimated in the same way. The Fourier transforms of the summands in the five terms are supported in $\{\vert\xi\vert\le8\}$, $\{\vert\xi\vert\le8\}$, $\{\vert\xi\vert\le2^{\max(k,i)+1}\cdot\frac85\}$, $\{\frac15 2^i\le\vert\xi\vert\le\frac95 2^i\}$ and $\{\frac15 2^k\le\vert\xi\vert\le\frac95 2^k\}$, respectively. Hence, only $j\le4$ contributes to $I_{21}$ and $I_{22}$, only $j\le\max(k,i)+2$ to $I_{23}$, and only $\vert j-i\vert\le2$, resp.\ $\vert j-k\vert\le2$, to $I_{24}$ and $I_{25}$. In each integral we then remove $P_j$ by Lemma~\ref{lem_bound_P} and apply \eqref{eq_transversal}, where the factor with the lower or comparable frequency is put into $L^1_{x^{(1)}}L^\infty_{x'}$. By Remark~\ref{rem_mixed} and Lemma~\ref{lem_anisotropic} this factor is bounded by $C\norm{f^0(\cdot,v)}_{B^s_{1,1}}$, resp.\ $C\norm{g^0(\cdot,u)}_{B^s_{1,1}}$.

More precisely, the first integral is bounded by
\begin{align*}
	I_{21}&\le C\iint_{\Omega_1}\sum_{j=1}^{4}2^{js}\norm{P_{\le0}f^0(\cdot,v)}_{L^1_{x^{(1)}}L^\infty_{x'}}\norm{P_{\le2}g^0(\cdot,u)}_{L^1_x}\d u\d v\\
	&\le C\norm{f^0}_{B_{1,1}^sL_{v}^{1}}\norm{g^0}_{B_{1,1}^sL_{v}^{1}}.
\end{align*}
With an analogous calculation, where now $P_{\le0}g^0$ is put into $L^1_{x^{(1)}}L^\infty_{x'}$, we obtain the same bound for $I_{22}$. For the third integral we obtain
\begin{align*}
	I_{23}&\le C\iint_{\Omega_1}\sum_{j=1}^{\infty}2^{js}\sum_{i\ge\max(1,j-4)}\ \sum_{\vert k-i\vert\le2}\norm{P_kf^0(\cdot,v)}_{L^1_{x^{(1)}}L^\infty_{x'}}\norm{P_ig^0(\cdot,u)}_{L^1_x}\d u\d v\\
	&\le C\iint_{\Omega_1}\norm{f^0(\cdot,v)}_{B^s_{1,1}}\sum_{i=1}^{\infty}2^{is}\norm{P_ig^0(\cdot,u)}_{L^1_x}\sum_{j\le i+4}2^{(j-i)s}\d u\d v\\
	&\le C\norm{f^0}_{B_{1,1}^sL_{v}^{1}}\norm{g^0}_{B_{1,1}^sL_{v}^{1}},
\end{align*}
since $s>0$ and hence $\sum_{j\le i+4}2^{(j-i)s}\le\frac{2^{4s}}{1-2^{-s}}$. For the fourth integral, the low-frequency factor $P_{\le i-3}f^0$ is put into $L^1_{x^{(1)}}L^\infty_{x'}$, which gives
\begin{align*}
	I_{24}&\le C\iint_{\Omega_1}\sum_{i=3}^{\infty}\sum_{\vert j-i\vert\le2}2^{js}\norm{P_{\le i-3}f^0(\cdot,v)}_{L^1_{x^{(1)}}L^\infty_{x'}}\norm{P_ig^0(\cdot,u)}_{L^1_x}\d u\d v\\
	&\le C\norm{f^0}_{B_{1,1}^sL_{v}^{1}}\norm{g^0}_{B_{1,1}^sL_{v}^{1}}.
\end{align*}
Finally, in the fifth integral the roles are interchanged. The low-frequency factor $P_{\le k-3}g^0$ is put into $L^1_{x^{(1)}}L^\infty_{x'}$ and the high-frequency factor $P_kf^0$ into $L^1_x$, which gives
\begin{align*}
	I_{25}&\le C\iint_{\Omega_1}\sum_{k=3}^{\infty}\sum_{\vert j-k\vert\le2}2^{js}\norm{P_kf^0(\cdot,v)}_{L^1_x}\norm{P_{\le k-3}g^0(\cdot,u)}_{L^1_{x^{(1)}}L^\infty_{x'}}\d u\d v\\
	&\le C\norm{f^0}_{B_{1,1}^sL_{v}^{1}}\norm{g^0}_{B_{1,1}^sL_{v}^{1}}.
\end{align*}
This completes the proof of Lemma~\ref{lem: bilinear_free}.
\end{proof}

We now extend the bilinear estimate from free solutions to general elements of $\mathcal{N}^s$. Let $f\in\mathcal{N}^s$ and $f_0\defeq f(0)$. By \eqref{eq_duhamel_N}, we have
\begin{align}\notag
	f(t,x,v)=f_0(x-tv,v)+\int_0^t(Af)(r,x-(t-r)v,v)\d r\qquad\text{for a.e. }(t,x,v),
\end{align}
and by definition
\begin{align*}
	\norm{f}_{\mathcal{N}^s} = \norm{f_0}_{\mathcal{X}^s}+\int_0^\infty\norm{(Af)(r)}_{\mathcal{X}^s}\d r.
\end{align*}
We define
\begin{align*}
	F_r(y,v) \defeq (Af)(r, y+rv, v), \qquad r > 0,
\end{align*}
i.e.\ $F_r=U(-r)(Af)(r)$. Since $x-(t-r)v=(x-tv)+rv$, we have $(Af)(r,x-(t-r)v,v)=F_r(x-tv,v)$, and therefore
\begin{align}\label{eq:superposition}
	f(t,x,v) = f_0(x-tv,v) + \int_0^\infty \mathbf{1}_{\{r<t\}}\, F_r(x-tv,v) \d r
	\qquad \text{for a.e. } (t,x,v). 
\end{align}
This means that functions in $\mathcal{N}^s$ can be represented as superpositions of free transport flows, namely of the free flow of $f_0$ and the free flows of $F_r$, $r>0$, where the latter are switched on at time $r$.

 \begin{lem}\label{lem: free_flow}
 $\norm{F_r}_{\mathcal{X}^s} = \norm{(Af)(r)}_{\mathcal{X}^s}$ for a.e.\ $r>0$, and
 \begin{align}\label{eq:superposition_norm}
 	\norm{f_0}_{\mathcal{X}^s} + \int_0^\infty \norm{F_r}_{\mathcal{X}^s} \d r = \norm{f}_{\mathcal{N}^s}.
 \end{align}
 \end{lem}
 
 \begin{proof}
 Since $U(-r)$ is an isometry on $\mathcal{X}^s$, we have $\norm{F_r}_{\mathcal{X}^s}=\norm{U(-r)(Af)(r)}_{\mathcal{X}^s}=\norm{(Af)(r)}_{\mathcal{X}^s}$ for a.e.\ $r>0$. Integrating in $r$ and adding $\norm{f_0}_{\mathcal{X}^s}$ gives \eqref{eq:superposition_norm} by the definition of the norm of $\mathcal{N}^s$.
 \end{proof}
 
 \begin{lem}\label{lem: bilinear}
 Let $s \geq d-1$. There exists $C >0$ such that for all $f,g \in \mathcal{N}^s$ the collision operator satisfies $Q(f,g) \in L_t^1\mathcal{X}^s$ and $\norm{Q(f,g)}_{L_t^1\mathcal{X}^s} \leq C\norm{f}_{\mathcal{N}^s}\norm{g}_{\mathcal{N}^s}$.
 \end{lem}
 
 \begin{proof}
 By Lemma \ref{lem: free_flow} we write
 \begin{align*}
 	f(t,x,v) &= f_0(x-tv,v) + \int_0^\infty \mathbf{1}_{\{r<t\}} F_r(x-tv,v) \d r,\\
 	g(t,x,u) &= g_0(x-tu,u) + \int_0^\infty \mathbf{1}_{\{r'<t\}} G_{r'}(x-tu,u) \d r'.
 \end{align*}
 By $f^{\mathrm{free}}(t,x,v) \defeq f_0(x-tv,v)$, $f^r(t,x,v) \defeq F_r(x-tv,v)$ and $g^{\mathrm{free}}$, $g^{r'}$ respectively, we denote the corresponding free transport solutions.

 Applying the estimate from the proof of Lemma~\ref{lem: bilinear_free} with $\alpha=\beta=0$, which only involves $\norm{F_r}_{L^1_vL^1_{x^{(i)}}L^\infty_{x'}}$ and $\norm{G_{r'}}_{L^1_{x,u}}$, which are integrable in $r$ resp.\ $r'$ by Lemma~\ref{lem_anisotropic} and \eqref{eq:superposition_norm}, and is therefore invariant under replacing $F_r$, $G_{r'}$ by $\vert F_r\vert$, $\vert G_{r'}\vert$, shows that all integrals in the definition of $Q^+$ and $Q^-$ converge absolutely for a.e.\ $(t,x,v)$.
 Thus, by Fubini's theorem and the bilinearity of the collision operator
 \begin{align*}
 	Q(f,g)(t) &= Q(f^{\mathrm{free}},g^{\mathrm{free}})(t) + \int_0^\infty \mathbf{1}_{\{r<t\}} Q(f^r,g^{\mathrm{free}})(t) \d r + \int_0^\infty \mathbf{1}_{\{r'<t\}} Q(f^{\mathrm{free}},g^{r'})(t) \d r'\\
 	&+ \int_0^\infty \int_0^\infty \mathbf{1}_{\{r<t\}}\mathbf{1}_{\{r'<t\}} Q(f^r,g^{r'})(t) \d r \d r'.
 \end{align*}
 The Littlewood-Paley projections $P_j$ and the derivatives $D^{\alpha}$ act only in the spatial variable and therefore commute with the integration in $r$ and $r'$. Using Minkowski's inequality and $\norm{\mathbf{1}_{\{t>r\}}h}_{L_t^1\mathcal{X}^s} \leq \norm{h}_{L_t^1\mathcal{X}^s}$ we obtain
 \begin{align*}
 	\norm{Q(f,g)}_{L_t^1\mathcal{X}^s} &\leq \norm{Q(f^{\mathrm{free}},g^{\mathrm{free}})}_{L_t^1\mathcal{X}^s} + \int_0^\infty \norm{Q(f^r,g^{\mathrm{free}})}_{L_t^1\mathcal{X}^s} \d r + \int_0^\infty \norm{Q(f^{\mathrm{free}},g^{r'})}_{L_t^1\mathcal{X}^s} \d r'\\
 	&+ \int_0^\infty \int_0^\infty \norm{Q(f^r,g^{r'})}_{L_t^1\mathcal{X}^s} \d r \d r'.
 \end{align*}
 Since $f^{\mathrm{free}}, f^r, g^{\mathrm{free}}, g^{r'}$ are free transport solutions, Lemma \ref{lem: bilinear_free} yields
 \begin{align*}
 	\norm{Q(f,g)}_{L_t^1\mathcal{X}^s} &\leq C\left(\norm{f_0}_{\mathcal{X}^s} + \int_0^\infty \norm{F_r}_{\mathcal{X}^s} \d r\right)\left(\norm{g_0}_{\mathcal{X}^s} + \int_0^\infty \norm{G_{r'}}_{\mathcal{X}^s} \d r'\right)\\
 	&= C\norm{f}_{\mathcal{N}^s}\norm{g}_{\mathcal{N}^s},
 \end{align*}
 where we used \eqref{eq:superposition_norm} in the last step.
 \end{proof}
 
 \section{Global well-posedness for small initial data}
  We now combine the bilinear estimate with a fixed-point argument in the space $\mathcal{N}^s$ in order to prove Theorem \ref{thm_glob_well_posedness}.

\begin{definition}\label{def_mild}
Let $s\ge d-1$ and $f^0\in\mathcal{X}^s$. A function $f\in\mathcal{N}^s$ is called a (global) mild solution of \eqref{IVP} if $f(0)=f^0$ and $Af=Q(f,f)$ in $L^1_t\mathcal{X}^s$.
\end{definition}
By Lemma~\ref{lem: bilinear} and its proof, for $f\in\mathcal{N}^s$ the integrals defining $Q(f,f)$ converge absolutely for a.e.\ $(t,x,v)$ and $Q(f,f)\in L^1_t\mathcal{X}^s$, so that the condition in Definition~\ref{def_mild} is an identity between two elements of $L^1_t\mathcal{X}^s$. In particular, every mild solution satisfies $Q(f,f)\in L^1_t\mathcal{X}^s$. By Duhamel's formula (\eqref{eq_duhamel_N}), $f\in\mathcal{N}^s$ with $f(0)=f^0$ is a mild solution if and only if
\begin{align*}
	f(t,x,v) = f^0(x-tv,v) + \int_0^t Q(f,f)(r,x-(t-r)v,v)\d r \qquad\text{for a.e. } (t,x,v).
\end{align*}
 
 \begin{proof}[Proof of Theorem \ref{thm_glob_well_posedness}]
 	 Let $f^0 \in \mathcal{X}^s$ be given. For $h \in \mathcal{N}^s$, define the mapping $\Phi(h) = f$ on $\mathcal{N}^s$, where $f$ is the solution to the linear transport equation
 	\begin{align*}
 		\partial_t f + v \cdot \nabla_x f &= Q(h,h),\\
 		f(0,x,v) &= f^0(x,v).
 	\end{align*}
 	More precisely, $f\in\mathcal{N}^s$ is given by $f(0)=f^0$ and $Af=Q(h,h)$, i.e.\ $f(t)=U(t)f^0+\int_0^tU(t-r)Q(h,h)(r)\d r$. This is well defined, since $Q(h,h)\in L^1_t\mathcal{X}^s$ by Lemma~\ref{lem: bilinear}.
 	Let $C>0$ be such that the estimate of Lemma~\ref{lem: bilinear} holds with constant $C$ for $Q$ and also for the gain term $Q^+$ and the loss term $Q^-$ separately. Such a constant exists, since the proofs of Lemma~\ref{lem: bilinear_free} and Lemma~\ref{lem: bilinear} treat the gain and the loss term separately.
 	By Lemma \ref{lem: bilinear} the collision operator satisfies
 	\begin{align*}
 		\|Q(h,h)\|_{L_t^1\mathcal{X}^s}
 		\le
 		C \|h\|_{\mathcal{N}^s}^2.
 	\end{align*}
 	Hence,
 	\begin{align*}
\|\Phi(h)\|_{\mathcal{N}^s}
=
\|f^0\|_{\mathcal{X}^s}
+
\|Q(h,h)\|_{L_t^1\mathcal{X}^s}
\le
\|f^0\|_{\mathcal{X}^s}
+
C\|h\|_{\mathcal{N}^s}^2.
 	\end{align*}
 	Choose $R=\frac{1}{4C}$ and assume  $\norm{f^0}_{\mathcal{X}^s} < \frac{3}{32C}\defeq\varepsilon_0$.
 	Then $\Phi$ maps the closed ball $B_R := \{ h \in \mathcal{N}^s : \|h\|_{\mathcal{N}^s} \le R \}$ into itself.
 	
 	Moreover, for $h_1,h_2\in B_R$,
 	\begin{align*}
 		\|Q(h_1,h_1)-Q(h_2,h_2)\|_{L_t^1\mathcal{X}^s} &=\|Q(h_1,h_1-h_2)+Q(h_1-h_2,h_2)\|_{L_t^1\mathcal{X}^s}\\
 		&\le
 		C(\|h_1\|_{\mathcal{N}^s}+\|h_2\|_{\mathcal{N}^s})
 		\|h_1-h_2\|_{\mathcal{N}^s}
 		\le
 		2CR \|h_1-h_2\|_{\mathcal{N}^s}.
 	\end{align*}
 
 Since $\Phi(h_1)-\Phi(h_2)$ has vanishing initial data and $A(\Phi(h_1)-\Phi(h_2))=Q(h_1,h_1)-Q(h_2,h_2)$, we obtain $\norm{\Phi(h_1)-\Phi(h_2)}_{\mathcal{N}^s}\le 2CR\norm{h_1-h_2}_{\mathcal{N}^s}$.
 Since $2CR=\frac12$, the mapping $\Phi$ is a contraction on $B_R$.
 The contraction mapping principle yields existence and uniqueness
 of a global mild solution $f\in\mathcal{N}^s$ in $B_R$.

 \emph{Unconditional uniqueness.} Let $g\in\mathcal{N}^s$ be any mild solution of \eqref{IVP} in the sense of Definition~\ref{def_mild} with $g(0)=f^0$. We show $g=f$. For $T_0\ge0$, $\delta>0$, $I=[T_0,T_0+\delta]$ and $h$ with $h(T_0)\in\mathcal{X}^s$, $Ah\in L^1(I,\mathcal{X}^s)$ and $h(t)=U(t-T_0)h(T_0)+\int_{T_0}^tU(t-r)(Ah)(r)\d r$ for $t\in I$ we set $\norm{h}_{\mathcal{N}^s(I)}\defeq\norm{h(T_0)}_{\mathcal{X}^s}+\norm{Ah}_{L^1(I,\mathcal{X}^s)}$. Extending $h$ by $\tilde h(t)\defeq h(T_0+t)$ for $t\in[0,\delta]$ and by the free transport of $h(T_0+\delta)$ for $t>\delta$, we obtain $\tilde h\in\mathcal{N}^s$ with $\norm{\tilde h}_{\mathcal{N}^s}=\norm{h}_{\mathcal{N}^s(I)}$. Since $Q$ acts pointwise in $t$, Lemma~\ref{lem: bilinear} yields
 \begin{align*}
 	\norm{Q(h_1,h_2)}_{L^1(I,\mathcal{X}^s)}\le C\norm{h_1}_{\mathcal{N}^s(I)}\norm{h_2}_{\mathcal{N}^s(I)}.
 \end{align*}
 Let $T^\ast\defeq\sup\{T\ge0:\ f=g \text{ on } [0,T]\}$ and assume $T^\ast<\infty$. Since $f,g\in C([0,\infty),\mathcal{X}^s)$ by the definition of $\mathcal{N}^s$, we have $f=g$ on $[0,T^\ast]$. By Lemma~\ref{lem_sup_bound}, $\norm{f(T^\ast)}_{\mathcal{X}^s}\le\norm{f}_{\mathcal{N}^s}\le R$. Since $Af,Ag\in L^1_t\mathcal{X}^s$, we can choose $\delta>0$ such that $\norm{Af}_{L^1(I,\mathcal{X}^s)}+\norm{Ag}_{L^1(I,\mathcal{X}^s)}\le\frac{1}{4C}$ for $I=[T^\ast,T^\ast+\delta]$. Then $\norm{f}_{\mathcal{N}^s(I)}+\norm{g}_{\mathcal{N}^s(I)}\le2R+\frac1{4C}=\frac{3}{4C}$. The difference $w\defeq f-g$ satisfies $w(T^\ast)=0$ and $Aw=Q(f,f)-Q(g,g)=Q(f,w)+Q(w,g)$, hence
 \begin{align*}
 	\norm{w}_{\mathcal{N}^s(I)}=\norm{Aw}_{L^1(I,\mathcal{X}^s)}\le C\big(\norm{f}_{\mathcal{N}^s(I)}+\norm{g}_{\mathcal{N}^s(I)}\big)\norm{w}_{\mathcal{N}^s(I)}\le\tfrac34\norm{w}_{\mathcal{N}^s(I)}.
 \end{align*}
 Thus $w=0$ on $I$, contradicting the definition of $T^\ast$. Hence $T^\ast=\infty$ and $g=f$. In particular, the uniqueness statement in Theorem~\ref{thm_glob_well_posedness} holds without restriction to the ball $B_R$.

 \emph{Nonnegativity.} It remains to show $f\ge0$ if $f^0\ge0$. This follows from a standard approximation procedure as in \cite{DL1989b} or \cite{GHN2026}. We write the collision operator as $Q(g,h)=Q^+(g,h)-g\,L(h)$, i.e.\ $Q^-(g,h)=g\,L(h)$, with
 \begin{align*}
 	L(h)(t,x,v)\defeq\int_{\R^d}\int_{\omega\in S_+^{d-1}}h(t,x,u)B(v-u,\omega)\d\omega\d u.
 \end{align*}
  By the choice of $C$ we have $\norm{Q^+(g,h)}_{L^1_t\mathcal{X}^s}\le C\norm{g}_{\mathcal{N}^s}\norm{h}_{\mathcal{N}^s}$ and $\norm{g\,L(h)}_{L^1_t\mathcal{X}^s}\le C\norm{g}_{\mathcal{N}^s}\norm{h}_{\mathcal{N}^s}$ for all $g,h\in\mathcal{N}^s$.

 The proof relies on the following consequence of the transversality. For $h\in\mathcal{N}^s$ we have
 \begin{align}\label{eq_line_integral}
 	\int_0^\infty L(\vert h\vert)(r,x+rv,v)\d r\le C\norm{h}_{\mathcal{N}^s}\qquad\text{for a.e. }(x,v).
 \end{align}
 Indeed, let $v\in\R^d$ and let $\psi\in L^1(\R^d)$ be nonnegative. By Lemma~\ref{lem: free_flow} and $B(v-u,\omega)\lesssim\vert v-u\vert$,
 \begin{align*}
 	&\int_{\R^d}\psi(x)\int_0^\infty L(\vert h\vert)(r,x+rv,v)\d r\d x\\
 	&\lesssim\int_{\R^d}\vert v-u\vert\int_0^\infty\int_{\R^d}\psi(x)\Big(\big\vert h(0)(x+r(v-u),u)\big\vert+\int_0^\infty\big\vert F_{r'}(x+r(v-u),u)\big\vert\d r'\Big)\d x\d r\d u,
 \end{align*}
 where $F_{r'}$ is defined as in Lemma~\ref{lem: free_flow} with $h$ in place of $f$. The argument leading to \eqref{eq_transversal}, applied to $h(0)(\cdot,u)$, resp.\ $F_{r'}(\cdot,u)$, and $\psi$ with the roles of $v$ and $u$ interchanged, together with Lemma~\ref{lem_anisotropic} and \eqref{eq:superposition_norm}, bounds the right-hand side by $C\norm{h}_{\mathcal{N}^s}\norm{\psi}_{L^1_x}$. Since $\psi$ is arbitrary, \eqref{eq_line_integral} follows. All integrands in this argument are nonnegative, so that Tonelli's theorem justifies the interchange of the integrals without any a priori integrability. Since $\vert L(h)\vert\le L(\vert h\vert)$, \eqref{eq_line_integral} implies in particular that $r\mapsto L(h)(r,x+rv,v)$ is integrable on $(0,\infty)$ for a.e.\ $(x,v)$.

 We now construct nonnegative approximations of $f$. Let $g_0\defeq0$ and, for $k\ge0$, let $g_{k+1}\in\mathcal{N}^s$ be the solution of the linear problem
 \begin{align*}
 	Ag_{k+1}=Q^+(g_k,g_k)-g_{k+1}L(g_k),\qquad g_{k+1}(0)=f^0.
 \end{align*}
 We show by induction that $g_{k+1}$ exists and satisfies $\norm{g_{k+1}}_{\mathcal{N}^s}\le\frac R2$ and $g_{k+1}\ge0$. Assume that $\norm{g_k}_{\mathcal{N}^s}\le\frac R2$ and $g_k\ge0$, which is true for $k=0$.

 First, the map $g\mapsto U(t)f^0+\int_0^tU(t-r)\big(Q^+(g_k,g_k)-g\,L(g_k)\big)(r)\d r$ is a contraction on $\mathcal{N}^s$ with constant $\frac{CR}{2}=\frac18$. Its fixed point is $g_{k+1}$, and
 \begin{align*}
 	\norm{g_{k+1}}_{\mathcal{N}^s}\le\norm{f^0}_{\mathcal{X}^s}+\frac{CR^2}{4}+\frac18\norm{g_{k+1}}_{\mathcal{N}^s}.
 \end{align*}
 Hence, $\norm{g_{k+1}}_{\mathcal{N}^s}\le\frac87\big(\varepsilon_0+\frac{1}{64C}\big)\le\frac R2$, since $\varepsilon_0\le\frac{3}{32C}$.

 Second, since $g_k\ge0$, we have $Q^+(g_k,g_k)\ge0$ and $L(g_k)\ge0$. By \eqref{eq_duhamel_N}, for a.e.\ $(x,v)$ the function $G(t)\defeq g_{k+1}(t,x+tv,v)$ satisfies for a.e.\ $t>0$
 \begin{align*}
 	G(t)=f^0(x,v)+\int_0^t\Big(Q^+(g_k,g_k)(r,x+rv,v)-L(g_k)(r,x+rv,v)\,G(r)\Big)\d r.
 \end{align*}
 Since $Q^+(g_k,g_k)$ and $g_{k+1}L(g_k)$ lie in $L^1_t\mathcal{X}^s\subset L^1_{t,x,v}$, Fubini's theorem and the change of variables $x\mapsto x-tv$ show that both terms in the integral are integrable in $r$ for a.e.\ $(x,v)$. By \eqref{eq_line_integral}, also $r\mapsto L(g_k)(r,x+rv,v)$ is integrable for a.e.\ $(x,v)$. Hence, after modification on a null set, $G$ is absolutely continuous and solves a linear ordinary differential equation, whose unique solution is
 \begin{align*}
 	G(t)&=e^{-\int_0^tL(g_k)(\sigma,x+\sigma v,v)\d\sigma}f^0(x,v)\\
 	&+\int_0^te^{-\int_r^tL(g_k)(\sigma,x+\sigma v,v)\d\sigma}Q^+(g_k,g_k)(r,x+rv,v)\d r.
 \end{align*}
 Thus $G\ge0$, i.e.\ $g_{k+1}\ge0$.

 Finally, we show that $(g_k)$ converges to $f$. For $k\ge1$ the difference $g_{k+1}-g_k$ has vanishing initial data and
 \begin{align*}
 	A(g_{k+1}-g_k)&=Q^+(g_k,g_k-g_{k-1})+Q^+(g_k-g_{k-1},g_{k-1})\\
 	&-(g_{k+1}-g_k)L(g_k)-g_kL(g_k-g_{k-1}).
 \end{align*}
 Since all $g_k$ lie in the ball of radius $\frac R2$, this yields
 \begin{align*}
 	\norm{g_{k+1}-g_k}_{\mathcal{N}^s}\le\frac38\norm{g_k-g_{k-1}}_{\mathcal{N}^s}+\frac18\norm{g_{k+1}-g_k}_{\mathcal{N}^s},
 \end{align*}
 i.e.\ $\norm{g_{k+1}-g_k}_{\mathcal{N}^s}\le\frac37\norm{g_k-g_{k-1}}_{\mathcal{N}^s}$. Hence, $(g_k)$ converges in $\mathcal{N}^s$ to some $g$. Passing to the limit in the linear problem gives $g(0)=f^0$ and $Ag=Q^+(g,g)-g\,L(g)=Q(g,g)$, i.e.\ $g$ is a mild solution. By Lemma~\ref{lem_sup_bound}, $g_k(t)\to g(t)$ in $\mathcal{X}^s\subset L^1_{x,v}$ for every $t\ge0$, so that $g\ge0$. By the unconditional uniqueness shown above, $f=g\ge0$.

 \emph{Properties (i)--(iii).} (i) By the definition of $\mathcal{N}^s$ we have $f\in C([0,\infty),\mathcal{X}^s)$, and Lemma~\ref{lem_sup_bound} yields
 \begin{align*}
 	\sup_{t\ge0}\norm{f(t)}_{\mathcal{X}^s}\le\norm{f}_{\mathcal{N}^s}=\norm{f^0}_{\mathcal{X}^s}+\norm{Q(f,f)}_{L^1_t\mathcal{X}^s}.
 \end{align*}

 (ii) Let $\tilde f\in B_R$ be the solution constructed above with nonnegative initial data $\tilde f^0$, $\norm{\tilde f^0}_{\mathcal{X}^s}<\varepsilon_0$. The difference $f-\tilde f$ has initial data $f^0-\tilde f^0$ and satisfies $A(f-\tilde f)=Q(f,f-\tilde f)+Q(f-\tilde f,\tilde f)$. As in the contraction argument, Lemma~\ref{lem: bilinear} and $\norm{f}_{\mathcal{N}^s}+\norm{\tilde f}_{\mathcal{N}^s}\le2R$ yield
 \begin{align*}
 	\norm{f-\tilde f}_{\mathcal{N}^s}\le\norm{f^0-\tilde f^0}_{\mathcal{X}^s}+2CR\norm{f-\tilde f}_{\mathcal{N}^s}=\norm{f^0-\tilde f^0}_{\mathcal{X}^s}+\frac12\norm{f-\tilde f}_{\mathcal{N}^s},
 \end{align*}
 hence $\norm{f-\tilde f}_{\mathcal{N}^s}\le2\norm{f^0-\tilde f^0}_{\mathcal{X}^s}$. Since
 \begin{align*}
 	\norm{f-\tilde f}_{\mathcal{N}^s}=\norm{f^0-\tilde f^0}_{\mathcal{X}^s}+\norm{Q(f,f)-Q(\tilde f,\tilde f)}_{L^1_t\mathcal{X}^s},
 \end{align*}
 this gives $\norm{Q(f,f)-Q(\tilde f,\tilde f)}_{L^1_t\mathcal{X}^s}\le\norm{f^0-\tilde f^0}_{\mathcal{X}^s}$. Moreover, Lemma~\ref{lem_sup_bound} gives $$\sup_{t\ge0}\norm{f(t)-\tilde f(t)}_{\mathcal{X}^s}\le\norm{f-\tilde f}_{\mathcal{N}^s}\le2\norm{f^0-\tilde f^0}_{\mathcal{X}^s}.$$ Adding both estimates yields (ii). The nonnegativity of the initial data is not used here. It only ensures that $\tilde f$ is a solution in the sense of Theorem~\ref{thm_glob_well_posedness}.

 (iii) Let $F_r$ be defined as in Lemma~\ref{lem: free_flow}. Since $\norm{F_r}_{\mathcal{X}^s}=\norm{Q(f,f)(r)}_{\mathcal{X}^s}$ is integrable in $r$, the function $f_\infty\defeq f(0)+\int_0^\infty F_r\d r$ belongs to $\mathcal{X}^s$. By \eqref{eq:superposition}, $U(-t)f(t)=f(0)+\int_0^tF_r\d r$. Since $U(t)$ is an isometry on $\mathcal{X}^s$  (see Section~\ref{sec_function_spaces}), we obtain
 \begin{align*}
 	\norm{f(t)-U(t)f_\infty}_{\mathcal{X}^s}=\Big\Vert\int_t^\infty F_r\d r\Big\Vert_{\mathcal{X}^s}\le\int_t^\infty\norm{Q(f,f)(r)}_{\mathcal{X}^s}\d r\longrightarrow0\qquad(t\to\infty).
 \end{align*}
  \end{proof}

\section{Acknowledgements}
We would like to thank the referees for their careful reading and detailed feedback which
led to substantial improvements in the clarity and presentation of our paper.\\
Funded by the Deutsche Forschungsgemeinschaft (DFG, German Research Foundation), Project-ID 317210226, SFB 1283.

\bibliographystyle{alpha}
\bibliography{Literatur}

\begin{thebibliography}{HLPZ26}

\bibitem[ACG10]{ACG2010}
Ricardo~J. Alonso, Emanuel Carneiro, and Irene~M. Gamba.
\newblock Convolution inequalities for the {B}oltzmann collision operator.
\newblock {\em Comm. Math. Phys.}, 298(2):293--322, 2010.

\bibitem[AF03]{AF2003}
Robert~A. Adams and John J.~F. Fournier.
\newblock {\em Sobolev spaces}, volume 140 of {\em Pure and Applied Mathematics
  (Amsterdam)}.
\newblock Elsevier/Academic Press, Amsterdam, second edition, 2003.

\bibitem[Ars11]{A2011}
Diogo Ars\'enio.
\newblock On the global existence of mild solutions to the {B}oltzmann equation
  for small data in {$L^D$}.
\newblock {\em Comm. Math. Phys.}, 302(2):453--476, 2011.

\bibitem[BCP97]{BCP1997}
Dario Benedetto, Emanuele Caglioti, and Mario Pulvirenti.
\newblock A one dimensional {B}oltzmann equation with inelastic collisions.
\newblock {\em Rend. Sem. Mat. Fis. Milano}, 67:169--179, 1997.

\bibitem[Bou98]{B1988}
J.~Bourgain.
\newblock Refinements of {S}trichartz' inequality and applications to
  {$2$}d-{NLS} with critical nonlinearity.
\newblock {\em Internat. Math. Res. Notices}, (5):253--283, 1998.

\bibitem[BP01]{BP2001}
Dario Benedetto and Mario Pulvirenti.
\newblock On the one-dimensional {B}oltzmann equation for granular flows.
\newblock {\em M2AN Math. Model. Numer. Anal.}, 35(5):899--905, 2001.

\bibitem[CDP21]{CDP2021}
Thomas Chen, Ryan Denlinger, and Nataša Pavlovi\'c.
\newblock Small data global well-posedness for a {B}oltzmann equation via
  bilinear spacetime estimates.
\newblock {\em Arch. Ration. Mech. Anal.}, 240(1):327--381, 2021.

\bibitem[Cer88a]{C1988}
Carlo Cercignani.
\newblock {\em The {B}oltzmann equation and its applications}, volume~67 of
  {\em Applied Mathematical Sciences}.
\newblock Springer-Verlag, New York, 1988.

\bibitem[Cer88b]{CC1988}
Carlo Cercignani.
\newblock Small data existence for the {E}nskog equation in {$L^1$}.
\newblock {\em J. Statist. Phys.}, 51(1-2):291--297, 1988.

\bibitem[CH18]{CH2018}
Timothy Candy and Sebastian Herr.
\newblock Transference of bilinear restriction estimates to quadratic variation
  norms and the {D}irac-{K}lein-{G}ordon system.
\newblock {\em Anal. PDE}, 11(5):1171--1240, 2018.

\bibitem[CP96]{CP1996}
Fran\c{c}ois Castella and Beno\^{\i}t Perthame.
\newblock Estimations de {S}trichartz pour les \'{e}quations de transport
  cin\'{e}tique.
\newblock {\em C. R. Acad. Sci. Paris S\'{e}r. I Math.}, 322(6):535--540, 1996.

\bibitem[CSZ23]{CSZ2023}
Xuwen Chen, Shunlin Shen, and Zhifei Zhang.
\newblock Sharp global well-posedness and scattering of the {B}oltzmann
  equation.
\newblock {\em https://arxiv.org/abs/2311.02008v1}, 2023.

\bibitem[CSZ24]{CSZ2024}
Xuwen Chen, Shunlin Shen, and Zhifei Zhang.
\newblock Well/ill-posedness of the {B}oltzmann equation with soft potential.
\newblock {\em Comm. Math. Phys.}, 405(12):Paper No. 283, 51, 2024.

\bibitem[DL89]{DL1989b}
Ronald~J. DiPerna and Pierre-Louis Lions.
\newblock On the {C}auchy problem for {B}oltzmann equations: global existence
  and weak stability.
\newblock {\em Ann. of Math. (2)}, 130(2):321--366, 1989.

\bibitem[DLX16]{DLX2016}
Renjun Duan, Shuangqian Liu, and Jiang Xu.
\newblock Global well-posedness in spatially critical {B}esov space for the
  {B}oltzmann equation.
\newblock {\em Arch. Ration. Mech. Anal.}, 220(2):711--745, 2016.

\bibitem[FK00]{KF2000}
Damiano Foschi and Sergiu Klainerman.
\newblock Bilinear space-time estimates for homogeneous wave equations.
\newblock {\em Ann. Sci. \'Ecole Norm. Sup. (4)}, 33(2):211--274, 2000.

\bibitem[GHN26]{GHN2026}
Benjamin Gess, Sebastian Herr, and Anne Niesdroy.
\newblock Existence of martingale solutions to a stochastic kinetic model of
  chemotaxis.
\newblock {\em NoDEA Nonlinear Differential Equations Appl.}, 33(2):Paper No.
  52, 44, 2026.

\bibitem[Ha03]{H2003}
Seung-Yeal Ha.
\newblock {$L^1$} stability estimate for a one-dimensional {B}oltzmann equation
  with inelastic collisions.
\newblock {\em J. Differential Equations}, 190(2):621--642, 2003.

\bibitem[HJ17]{HJ2017}
Ling-Bing He and Jin-Cheng Jiang.
\newblock Well-posedness and scattering for the {B}oltzmann equations: soft
  potential with cut-off.
\newblock {\em J. Stat. Phys.}, 168(2):470--481, 2017.

\bibitem[HJ23]{HJ2023}
Ling-Bing He and Jin-Cheng Jiang.
\newblock On the {C}auchy problem for the cutoff {B}oltzmann equation with
  small initial data.
\newblock {\em J. Stat. Phys.}, 190(3):Paper No. 52, 25, 2023.

\bibitem[HJKL24]{HJKL2024}
Ling-Bing He, Jin-Cheng Jiang, Hung-Wen Kuo, and Meng-Hao Liang.
\newblock The {$L^p$} estimate for the gain term of the {B}oltzmann collision
  operator and its application.
\newblock {\em Arch. Ration. Mech. Anal.}, 248(6):Paper No. 112, 55, 2024.

\bibitem[HLPZ26]{HLZ2026}
Xinfeng Hu, Shuangqian Liu, Haoran Peng, and Yi~Zhou.
\newblock Low-regularity global well-posedness for the {B}oltzmann equation
  near vacuum.
\newblock {\em arXiv:2604.04526}, 2026.

\bibitem[IS84]{IS1984}
Reinhard Illner and Marvin Shinbrot.
\newblock The {B}oltzmann equation: global existence for a rare gas in an
  infinite vacuum.
\newblock {\em Comm. Math. Phys.}, 95(2):217--226, 1984.

\bibitem[KM93]{KM1993}
S.~Klainerman and M.~Machedon.
\newblock Space-time estimates for null forms and the local existence theorem.
\newblock {\em Comm. Pure Appl. Math.}, 46(9):1221--1268, 1993.

\bibitem[Ovc11]{O2011}
Evgeni~Y. Ovcharov.
\newblock Strichartz estimates for the kinetic transport equation.
\newblock {\em SIAM J. Math. Anal.}, 43(3):1282--1310, 2011.

\bibitem[Tao07]{T2007}
Terence Tao.
\newblock Scattering for the quartic generalised {K}orteweg-de {V}ries
  equation.
\newblock {\em J. Differential Equations}, 232(2):623--651, 2007.

\bibitem[Tri78]{T1978}
Hans Triebel.
\newblock {\em Interpolation theory, function spaces, differential operators}.
\newblock VEB Deutscher Verlag der Wissenschaften, Berlin, 1978.

\bibitem[Vil02]{V2002}
C\'edric Villani.
\newblock A review of mathematical topics in collisional kinetic theory.
\newblock In {\em Handbook of mathematical fluid dynamics, {V}ol. {I}}, pages
  71--305. North-Holland, Amsterdam, 2002.

\end{thebibliography}

\begin{flushleft}
	\small \normalfont
	\textsc{Anne Niesdroy\\
		Fakult\"at f\"ur Mathematik, Universit\"at Bielefeld\\
		33615 Bielefeld, Germany.}\\
	\texttt{\textbf{aniesdroy@math.uni-bielefeld.de}}
\end{flushleft}
\end{document}